\documentclass[submission]{FPSAC2025}
\usepackage{cleveref}

\newtheorem{thm}{Theorem}[section]
\newtheorem{lem}[thm]{Lemma}
\newtheorem{prop}[thm]{Proposition}
\newtheorem{defi}[thm]{Definition}
\newtheorem{coro}[thm]{Corollary}
\newtheorem{ex}[thm]{Example}

\newtheorem{remark}[thm]{Remark}

\newtheorem{question}[thm]{Question}
\newtheorem{conjecture}[thm]{Conjecture}

\newcommand{\defn}[1]{{\color{Blue} \bf#1}}

\usepackage{float}   %choose the emplacement of figures
\usepackage{subcaption}  %for subfigures

\DeclareMathOperator{\Tam}{Tam}

\usepackage{lipsum}
\usepackage{tikz}

\title[$(P,\phi)$-Tamari lattices]{$(P,\phi)$-Tamari lattices}

\author[A. Segovia]{Adrien Segovia\thanks{\href{mailto:adrien.segovia@gmail.com}{adrien.segovia@gmail.com}}
}

\address{Université du Québec à Montréal, LACIM, Montréal, Canada}

\received{\today}

\abstract{Given any poset $P$ and chain $\phi$ in $P$, we define the $(P,\phi)$-Tamari lattice. We study in depth these lattices and prove in particular that they are join-semidistributive, join-congruence uniform and left modular. 
 We prove that the lattices of higher torsion classes of the higher Auslander and Nakayama algebras of type $\mathbb{A}$ are examples of $(P,\phi)$-Tamari lattices and thus they inherit their properties.
 We also give general results related to left modular, extremal and congruence normal lattices.}  


\usepackage[backend=bibtex]{biblatex}
\begin{document}

\usetikzlibrary{quotes}

\maketitle
%% note that you DO NOT have to put your abstract here -- it is generated by \maketitle and the \abstract and \resume commands above

\section{Introduction}

The Tamari lattice is familiar to many in combinatorics. It is perhaps most simply described as the lattice of planar binary trees ordered by tree rotation. For our purposes, it is also important to point out that it arises as the lattice of torsion classes for the type $\mathbb{A}_n$ linearly oriented path algebra. Our goal is to present a certain combinatorial generalization of the Tamari lattice which we call the $(P,\phi)$-Tamari lattice, where $P$ is a poset and $\phi$ is a chain in $P$. We establish certain properties for these lattices, including join-semidistributivity, join-congruence uniformity, and left modularity.

The reason that we were inspired to formulate the definition of the $(P,\phi)$-Tamari lattices is that they include, as a very special case, certain lattices which recently appeared in representation theory. For any finite-dimensional algebra, its torsion classes ordered by inclusion form a lattice. These have attracted considerable interest \cite{Demonet_2023,thomas2021introduction}. One property in particular of interest is that they are all semidistributive. There is a generalization due to J\o{}rgensen of torsion classes \cite{jorgensen2016torsion} in the setting of higher homological algebra, known as $d$-torsion classes. 
 Recently the authors of \cite{AHJKPT23} showed that J\o{}rgensen's definition is equivalent to being closed under $d$-extensions and $d$-quotients, thus yielding a lattice ordered by inclusion on these $d$-torsion classes. They obtained a combinatorial description of the $d$-torsion classes for the higher Auslander and Nakayama algebras of type $\mathbb{A}$, which are the two main examples where we are able to compute in higher homological algebra. 
 They noticed that the lattices of $d$-torsion classes of these algebras are not semidistributive in general, but we prove that they are a special case of our construction. Thus they inherit the properties of the $(P,\phi)$-Tamari lattices.

While working on these lattices, we developed a new way to prove the left modularity of a lattice using edge-labellings (\cref{thmlabelling}), and gave necessary and sufficient conditions on the doublings of a congruence normal lattice to be extremal or left modular (\cref{propExtremality,thmCongruencenormal}). These results might be useful to others working on lattices.

In \cref{background} we give some background on posets and lattices. \cref{sectionGeneralLattice} presents new tools to study left modular, congruence normal and extremal lattices.
In \cref{sectionPphiTamari} we define and study the $(P,\phi)$-Tamari lattices. The generalities are given in \cref{sectionGenerality}, proving in particular that they are join-semidistributive (\cref{propJSD}) and in \cref{sectionCongruenceTopo} we prove that they are left modular (\cref{thmleftmodular}), join-congruence uniform (\cref{thmLW}) and we study their congruences. Finally, in \cref{sectionMainExamples} we give the main examples: \cref{subsectionChain} for $P$ a chain, and \cref{sectionAuslander,sectionNakayama} for the $d$-torsion classes of the higher Auslander and Nakayama algebras of type $\mathbb{A}$.

\acknowledgements{I am very grateful to my advisors Samuele Giraudo and Hugh Thomas for their help.}

\section{Background on posets and lattices}
\label{background}

We denote by $|E|$ the cardinality of a set $E$. For a positive integer $n$, denote $[n]:=\{1,2,\dots,n\}$. By $k<n$ we will mean $k\in \{0,1,\dots,n-1\}$.

If a partially ordered set (poset) $(P,\leq)$ has a minimum, it is denoted $\hat{0}$, and $\hat{1}$ for the maximum. The cover relations of $P$ are denoted $x\lessdot y$ and we say that $y$ \defn{covers} $x$ or $x$ \defn{is covered by} $y$. They form the set $E(P)$ of the edges of its Hasse diagram, which is draw with smaller elements at the bottom. The interval of $P$ between $x$ and $y$ is denoted by $[x,y]$. Its Möbius function is written $\mu(x,y)$.
A subset $C$ of $P$ is \defn{convex} if for all $x$ and $y$ in $C$, we have $[x,y]\subseteq C$.   An \defn{order ideal} of a poset $P$ is a subset $I$ such that for all $x,y\in P$, if $y\leq x$ and $x\in I$, then $y\in I$. Dually, an \defn{order filter} is a subset $F$ such that for all $x,y\in P$, if $y\geq x$ and $x\in F$, then $y\in F$. The order ideal generated by a subset $C$ is denoted by $I_P(C):=\{y\in P\mid \exists x\in C,\,y\leq x\}$. If $C=\{x\}$, $I_P(\{x\})$ is a \defn{principal order ideal} and is denoted by $I_P(x)$. 
The poset ordered by inclusion on the order ideals of a poset $P$ is denoted $J(P)$. The number of order ideals of $P$, which is also its number of order filters, is $|J(P)|$. 
We denote by \defn{$C_n$} the chain $0<1<\cdots <n-1$. We also call \defn{chains} the totally ordered subsets of $P$. Equivalently, they are the image of an injective order preserving map $\phi : \{0,1,\dots,n\}\rightarrow P$, and we will use $\phi: \phi(0)< \phi(1) <\cdots <\phi(n)$ to denote them.
In this case, the length of this chain is $n$. If we cannot add elements to a chain, it is called a \defn{maximal chain}. The \defn{length} of a poset, denoted $\ell(P)$, is the maximum length of a chain in $P$. The chains of length $\ell(P)$ are called \defn{longest chains}. The \defn{spine} of a poset is the subset of the elements that lie on any chain of longest length. A \defn{linear extension} of a poset $(P,\leq)$ is a total order $\prec$ on $P$ such that $x\leq y$ implies $x\prec y$.
If $P$ and $Q$ are two posets, then the \defn{direct product} $P\times Q$ is the poset on the cartesian product $P\times Q$ defined by $(x,y)\leq (x',y')$ if $x\leq x'$ and $y\leq y'$.

A \defn{lattice} $L$ is a poset such that any pair of elements $\{x,y\}$ admits a least upper bound, called the \defn{join} and written $x\vee y$, and a greatest lower bound, called the \defn{meet} and written $x\wedge y$. We will denote $\Tam_n$ the Tamari lattice of size $n$ which has cardinality $\frac{1}{n+1} \binom{2n}{n}$.
All posets and lattices considered in this extended abstract are assumed to be finite, thus the following definitions are given for a finite lattice $L$.   
A \defn{join-irreducible} $j\in L$ is an element that covers a unique element, denoted $j_*$, and a \defn{meet-irreducible} $m\in L$ is one that is covered by a unique element.
The sets of these elements are respectively denoted $JIrr(L)$ and $MIrr(L)$. An \defn{edge-labelling} of $L$ is a map $\gamma: E(L) \rightarrow P$ where $P$ is a poset.
It is well known that in a lattice, an element is always the join of the join-irreducibles below it, and it follows that $\ell(L)\leq |\mathrm{JIrr}(L)|$ (or $|\mathrm{MIrr}(L)|$). 

\begin{defi}
    A lattice $L$ is \defn{join-extremal} if $\ell(L)=|\mathrm{JIrr}(L)|$, \defn{meet-extremal} if $\ell(L)=|\mathrm{MIrr}(L)|$ and \defn{extremal} if it is both join and meet-extremal. 
\end{defi}

\begin{defi}
 A lattice $L$ is \defn{join-semidistributive (JSD)} if for all $x,y,z\in L$, we have $x\vee y=x\vee z \Longrightarrow x\vee (y\wedge z)=x\vee y$. It is \defn{meet-semidistributive} if for all $x,y,z\in L$, we have $x\wedge y=x\wedge z \Longrightarrow x\wedge (y\vee z)=x\vee y$. It is \defn{semidistributive (SD)} if it is both join and meet-semidistributive.   
\end{defi}

\begin{lem}[\cite{freese1995free}]
\label{lemJSDjoinmeetSD}
A lattice $L$ is SD if and only if $L$ is JSD and $|\mathrm{JIrr}(L)|=|\mathrm{MIrr}(L)|$ .
\end{lem}

\begin{lem}
\label{lemEquiJSD}
A lattice $L$ is JSD if and only if for all covers $b\lessdot c$, the set $I_L(c)\setminus I_L(b)$ has a minimum element. In this case, the minimum is a join-irreducible element.
\end{lem}

  Thus by \cref{lemEquiJSD}, if $L$ is JSD then $\gamma: E(L) \rightarrow \mathrm{JIrr}(L)$ that sends a cover $b\lessdot c$ to $\mathrm{min(}I_L(c)\setminus I_L(b))$ is a well defined edge-labelling.

\begin{defi}
     An element $a\in L$ is \defn{left modular} if for all $b<c$ in $L$, we have $(b\vee a) \wedge c = b\vee (a\wedge c)$. 
    A maximal chain made of left modular elements is called a maximal left modular chain. The lattice $L$ is left modular if there exists a maximal left modular chain. 
\end{defi}

We refer the reader to \cite{Wachs1996SHELLABLENC} for details related to the following topological notions. Denote $\overline{L}:=L\setminus \{\hat{0},\hat{1}\}$. The order complex $\Delta(P)$ of a poset $P$ is the simplicial complex of vertex set $P$ whose faces are the chains of $P$.  An \defn{$EL$-labelling} of a lattice $L$ is an edge-labelling such that in any interval, when reading the labels following the chains from bottom to top, there is a unique maximal increasing chain and the label word of the increasing chain lexicographically precedes the label word of any other maximal chains. If $L$ admits an $EL$-labelling, then its order complex $\Delta(\overline{L})$ is shellable and homotopy equivalent to a wedge of spheres.
Moreover, for all $x$ and $y$ in $L$ we have that $\mu(x,y)$ is given by the difference between the number of even length maximal decreasing chains and the number of odd length maximal decreasing chains. Thus if $L$ has at most one maximal decreasing chain in any interval, then $\mu (x,y)\in \{-1,0,1\}$ for all $x,y\in L$ and the order complex of each non-empty open interval $]x,y[$ has the homotopy type of either a sphere or a point. 
Any left modular lattice admits an $EL$-labelling (see \cref{rmkLiu}).

\begin{defi}
\label{defDoublementDay}
    Let $C$ be a convex subset of $L$. 
    The \defn{doubling} $L[C]$ is the subposet of $L\times C_2$ consisting of the subset
    $\big(I_L(C) \times \{0\}\big) \,\bigsqcup \,\left[\big((L\setminus I_L(C))\cup C\big) \times \{1\} \right]$. It is in fact a lattice.
\end{defi}

See \cref{fig:doublingsLeftmodular} for examples of the doubling construction.
 A subset $C$ is a \defn{lower pseudo-interval} if $C$ is a union of intervals sharing the same minimum element. A lattice $L$ is \defn{congruence normal} if it is obtained from the one element lattice by successive doublings of convex subsets \cite{day1994congruence}. If at each step we double a lower pseudo-interval then $L$ is \defn{join-congruence uniform} (often called \textit{lower-bounded} in the literature). If we use only doublings of intervals, $L$ is called \defn{congruence uniform}.

A (lattice) \defn{congruence} on $L$ is an equivalence relation $\equiv$  on $L$ such that for all $ x_1,x_2,y_1,y_2$ in $L$, we have that $x_1\equiv x_2$ and $y_1\equiv y_2$ imply both $x_1\wedge y_1 \equiv x_2\wedge y_2$ and $x_1\vee y_1 \equiv x_2\vee y_2$. We identify the congruences $\equiv$ with the set of join-irreducibles $j$ that they contract, meaning $j_*\equiv j$ (two congruences are the same exactly when these sets are the same). 
 Let $D$, called the \defn{join dependency relation}, be the binary relation on $\mathrm{JIrr}(L)$ defined by $p D q$ if $p\neq q$ and there exists $x\in L$ such that $p\leq q \vee x$ and $p \not \leq q_* \vee x$. 
A \defn{$D$-cycle} is a sequence of elements $a_1,\,a_2,\,\dots,a_k$ with $k\geq 2$ such that $a_1 D a_2 D \cdots D a_k D a_1$.

\begin{prop}[\cite{freese1995free}]
\label{propDayLW}
    The lattice $L$ is join-congruence uniform if and only if it contains no $D$-cycles. In this case, the congruences of $L$ correspond to the subsets $T \subseteq \mathrm{JIrr}(L)$ such that if $a D b$ and $b\in T$, then $a\in T$.
\end{prop}

\section{New results on lattices}
\label{sectionGeneralLattice}

In this section, we give some general results on lattices. First about left modular lattices. Here the edge-labellings are maps from $E(L)$ to $\mathbb{N}$.

\begin{defi}
\label{deflabelling}
Let $L$ be a lattice. Denote $\phi :\, \hat{0}=x_0 <\dots <x_k=\hat{1}$ a chain containing $\hat{0}$ and $\hat{1}$.
 For $j\in \mathrm{JIrr}(L)$, denote $\delta(j):=\mathrm{min}\{i\,|\,j\leq x_i\}$. For $m\in \mathrm{MIrr}$, denote $\beta(m):=\mathrm{max}\{i\,|\,m\geq x_{i-1}\}$. We define $4$ edge-labellings ; for a cover relation $b\lessdot c$  (see \cref{fig:sfig1})
 \begin{align*}
    \gamma_1 (b\lessdot c) &:= \mathrm{min}\{ \delta(j) \,|\, j\in \mathrm{JIrr}(L),\,j\leq c,\,j\not\leq b\}, \quad\quad\quad\,\,
    \gamma_2 (b\lessdot c) := \mathrm{min}\{i \,|\, b \vee x_i \geq c\}, \\
    \gamma_3 (b\lessdot c) &:= \mathrm{max}\{ \beta(m) \,|\, m\in \mathrm{MIrr}(L),\,m\geq b,\,m\not\geq c\},  \,\,\,
    \gamma_4 (b\lessdot c) := \mathrm{max}\{i \,|\, c \wedge x_{i-1} \leq b\}. 
\end{align*}
\end{defi}

\begin{thm}
\label{thmlabelling}
    For any lattice $L$, we have $\gamma_2 = \gamma_3 \leq \gamma_1 = \gamma_4$. Moreover $\gamma_2 = \gamma_4$ if and only if for all $i$, $x_i$ is left modular.
\end{thm}

\begin{remark}
\label{rmkLiu}
 S.-C. Liu proved in \cite{LiuLeftmodular}, for the case of a longest chain $\phi$, that $\gamma_2\leq \gamma_1$ and that if for all $i$, $x_i$ is left modular, then $\gamma_2=\gamma_1$. He also proved that the equal labellings that we obtain starting from a maximal left modular chain is an $EL$-labelling. What we add to the story, other than an easier proof, is a way to use these labellings to prove that lattices are left modular.
\end{remark}

\begin{figure}
\begin{subfigure}{.5\textwidth}
  \centering
  \begin{tikzpicture}
\begin{scope}[scale= 0.8]
      \draw (0,0)--(1,2)--(0,4)--(-1,2.8)--(-1,1.2)--(0,0);
      \draw (0,0) node[below]{$x_0$};
      \draw (0,0) node[red]{$\bullet$};
      \draw (0,4) node[above]{$x_3$};
      \draw (0,4) node[red]{$\bullet$};
      \draw (-1,2.8) node[left]{$x_2$};
      \draw (-1,2.8) node[red]{$\bullet$};
      \draw (-1,1.2) node[left]{$x_1$};
      \draw (-1,1.2) node[red]{$\bullet$};

      \draw (0.5,1) node[below right]{$3$};
    \draw (0.5,1) node[above left]{$3$};
      \draw (0.5,3) node[above right]{$1$};
    \draw (0.5,3) node[below left]{$1$};
      \draw (-0.6,3.4) node[above left]{$3$};
    \draw (-0.6,3.4) node[below right]{$3$};
      \draw (-1,2) node[left]{$2$};
    \draw (-1,2) node[right]{$2$};
      \draw (-0.5,0.6) node[below left]{$1$};
    \draw (-0.5,0.6) node[above right]{$1$};
\end{scope}

\begin{scope}[xshift = 4cm, scale=0.9]
    \draw (0,0)--(2,1)--(0,5)--(-2,1)--(0,0);
    \draw (0,0)--(0,1)--(1,3);
    \draw (0,1)--(-1,3);
    
    \draw (0,0) node[red]{$\bullet$};
    \draw (2,1) node[red]{$\bullet$};
    \draw (1,3) node[red]{$\bullet$};
    \draw (0,5) node[red]{$\bullet$};

    \draw (0,0) node[below]{$x_0$};
    \draw (2,1) node[right]{$x_1$};
    \draw (1,3) node[above right]{$x_2$};
    \draw (0,5) node[above]{$x_3$};

    \draw[blue, ultra thick] (-2,1)--(-1,3);
    \draw (-2,1) node[left,blue]{$b$};
    \draw (-1,3) node[above left,blue]{$c$};

    \draw (1,0.5) node[below right]{$1$};
    \draw (1,0.5) node[above left]{$1$};
    \draw (0,0.5) node[right]{$2$};
    \draw (0,0.5) node[left]{$2$};
    \draw (1.5,2) node[above right]{$2$};
    \draw (1.5,2) node[below left]{$2$};
    \draw (-1.5,2) node[above left, blue]{$2$};
    \draw (-1.5,2) node[below right, blue]{$1$};
    \draw (0.5,2) node[above left]{$1$};
    \draw (0.5,2) node[below right]{$1$};
    \draw (-0.5,4) node[above left]{$1$};
    \draw (-0.5,4) node[below right]{$1$};
    \draw (0.5,4) node[above right]{$3$};
    \draw (0.5,4) node[below left]{$3$};
    \draw (-0.5,2) node[above right]{$3$};
    \draw (-0.5,2) node[below left]{$3$};
    \draw (-1,0.5) node[above right]{$3$};
    \draw (-1,0.5) node[below left]{$3$};
\end{scope}
\end{tikzpicture}
  
  \caption{$\gamma_2$ at the right of an edge, $\gamma_4$ at the left.\\
 On the right $x_1$ is not left modular \\
 because of $b\lessdot c$.}
  \label{fig:sfig1}
\end{subfigure}%
\begin{subfigure}{.5\textwidth}
  \centering
  \begin{tikzpicture}
    \begin{scope}[scale=0.7]
     \draw (0,0)--(2,2)--(2,4)--(0,6)--(-2,4)--(-2,2)--(0,0);
     \draw (-2,2)--(2,4);
     \draw (2,2)--(-2,4);
     \draw[blue, ultra thick] (0,0)--(-2,2)--(-2,4)--(0,6);
     \draw (-2.5,3) node[left,blue]{$\phi$};
     \draw[red] (2,3) ellipse (0.6cm and 1.7cm);
     \draw (3,2) node[red]{$C$};
     \draw (2,2) node[red]{$\bullet$};
     \draw (2,4) node[red]{$\bullet$};
     \end{scope}
    \end{tikzpicture}
  \caption{The elements of the blue maximal chain $\phi$ are all comparable to at least one element in $C$.}
  \label{fig:sfig2}
\end{subfigure}
\caption{}
\label{fig:fig}
\end{figure}

 Using this approach, we get a simpler proof of the following.

\begin{coro}[Theorem $1.4$ of \cite{Thomas_2019}]
\label{CorExtremalSD}
Semidistributive extremal lattices are left modular.
\end{coro}

\begin{proof}
Since $L$ is extremal, the choice of any longest chain $\phi: \,\hat{0}=x_0 \lessdot \dots \lessdot x_n=\hat{1}$ gives a numbering of the join and meet-irreducibles $j_1,j_2,\dots,j_n$ and $m_1,m_2,\dots,m_n$ such that 
$x_i=j_1\vee \dots \vee j_i = m_{i+1}\wedge \dots \wedge m_n$.
 The semidistributivity condition gives two equal labellings of the cover relations $\gamma (b\lessdot c) := i$ if $\mathrm{min}(I_L(c)\setminus I_L(b)) =j_i$ or equivalently if $\mathrm{max}(F_L(b)\setminus F_L(c)) =m_i$. These two labellings are respectively $\gamma_1$ and $\gamma_3$. Thus $\gamma_1=\gamma_3$ and using \cref{thmlabelling} we obtain that $L$ is left modular.
\end{proof}

We now turn our attention to congruence normal lattices, characterizing those that are extremal or left modular by necessary and sufficient conditions on each doubling.

\begin{prop}
\label{propExtremality}
     Let $L$ be a congruence normal lattice. Then $L$ is join-extremal if and only if it is join-congruence uniform and at each doubling step we double a lower pseudo-interval that contains an element of the spine. 
     The lattice $L$ is extremal if and only if it is congruence uniform and at each doubling step we double an interval that contains an element of the spine.
\end{prop}

In the sequel $C$ is a convex subset of $L$ and smaller and bigger refer to weak relations.

\begin{defi}
Let us call the \defn{heart} of $C$, written $H(C)$, the set of elements of $C$ that are smaller than all the maximal elements of $C$ and bigger than all the minimal elements of $C$.
\end{defi}

\begin{thm}
\label{thmCongruencenormal}
Let $L$ be a congruence normal lattice. Then $L$ is left modular if and only if at each doubling step by $C$ we have that $H(C)$ has an element that lies on a maximal left modular chain. 
\end{thm}

\begin{figure}
    \centering
    \begin{tikzpicture}
    \begin{scope}[scale=0.9]
    \draw (0,0)--(1,1)--(0,2)--(-1,1)--(0,0);
    \draw[ultra thick, red] (-1,1)--(0,0)--(1,1);
    \draw (0,0) node[blue]{$\bullet$};
    \draw (0,2) node[blue]{$\bullet$};
    \draw (1,1) node[blue]{$\bullet$};
    \draw (-1,1) node[blue]{$\bullet$};
    \draw[thick] (0,0) circle (0.3);

    \begin{scope}[xshift=3cm, yshift= 1.5cm]
    \draw (0,0)--(1,1)--(0,2)--(-1,1)--(0,0);  
    \draw (-1,1)--(-1,-0.5)--(0,-1.5)--(1,-0.5)--(1,1);
    \draw[ultra thick, red] (1,-0.5)--(1,1);
    \draw (0,0) node[blue]{$\bullet$};
    \draw (0,2) node[blue]{$\bullet$};
    \draw (1,1) node[blue]{$\bullet$};
    \draw (-1,1) node[blue]{$\bullet$};
    \draw (0,-1.5)--(0,0);
    \draw (0,-1.5) node[blue]{$\bullet$};
    \draw[thick] (1,0.25) ellipse (0.4cm and 1.1cm);
    \end{scope}

    \begin{scope}[xshift=6.5cm, yshift= 1.5cm]
     \draw (-1,1)--(0,0)--(1,1);  
    \draw (0,0)--(0,-1.5);
    \draw (-1,1)--(-1,-0.5)--(0,-1.5)--(1,-0.5)--(1,1)--(2,2)--(2,0.5)--(1,-0.5);   
    \draw (-1,1)--(1,3)--(2,2);
    \draw[ultra thick,red] (1,1)--(1,-0.5)--(2,0.5);
    \draw (0,0) node[blue]{$\bullet$};
    \draw (0,-1.5) node[blue]{$\bullet$};
    \draw (1,1) node[blue]{$\bullet$};
    \draw (2,2) node[blue]{$\bullet$};
    \draw (1,3) node[blue]{$\bullet$};
    \draw[thick] (1,-0.5) circle (0.3);
    \end{scope}

    \begin{scope}[xshift=10.5cm, yshift= 1.5cm]
    \draw (-1,1)--(0,0)--(1,1); 
    \draw (0,0)--(0,-1.5);
    \draw (-1,1)--(-1,-0.5)--(0,-1.5)--(1,-0.5);
    \draw (1,1)--(1,-0.5)--(1.8,0.3);
    \draw (1.5,2.2)--(1.5,0.7)--(2.3,1.5);
    \draw (1,1)--(1.5,2.2);  
    \draw (1,-0.5)--(1.5,0.7);  
    \draw (1.8,0.3)--(2.3,1.5);  
    \draw (1.5,2.2)--(2.3,3);
    \draw (2.3,1.5)--(2.3,3);
    \draw (-1,1)--(1.3,4)--(2.3,3);  
    \draw (1.5,2.2) node[blue]{$\bullet$};
    \draw (0,-1.5) node[blue]{$\bullet$};
    \draw (2.3,3) node[blue]{$\bullet$};
    \draw (1.3,4) node[blue]{$\bullet$};
    \end{scope}
    \end{scope}
    \end{tikzpicture}
    \caption{We represent $3$ successive doublings. The left modular elements are the blue dots. The thick red edges form the convex subsets $C$ that we double and we circled the elements of $H(C)$.}
    \label{fig:doublingsLeftmodular}
\end{figure}
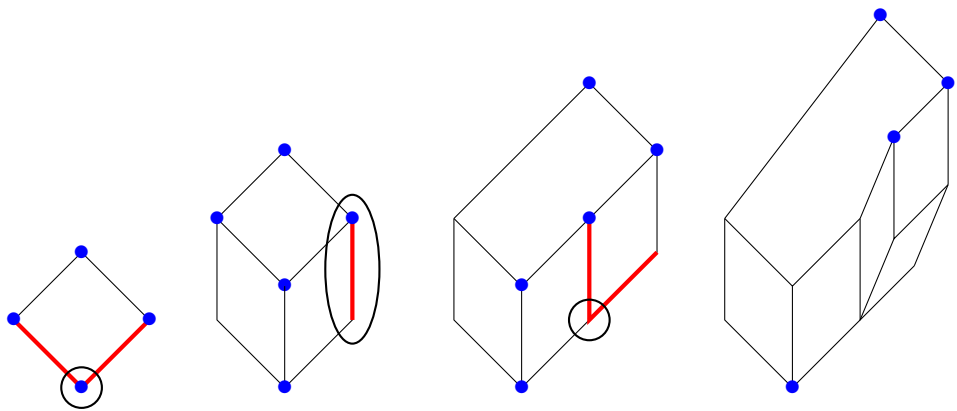

We give the two main ingredients (\cref{lemmeTechniqueMod,lemChainemaxIncomparable}) that lead to a proof of \cref{thmCongruencenormal}, that are interesting in their own right. 
See \cref{fig:doublingsLeftmodular} for an example of \cref{lemmeTechniqueMod} and \cref{fig:sfig2} for a counter-example to \cref{lemChainemaxIncomparable} when $L$ is not a lattice.

\begin{lem}
\label{lemmeTechniqueMod}
We have $(b,0)\in L[C]$ is left modular if and only if $b\in L$ is left modular and $b$ is smaller than all the maximal elements of $C$.
We have $(b,1)\in L[C]$ is left modular if and only if $b\in L$ is left modular and $b$ is bigger than all the minimal elements of $C$.
Thus $(b,0)\lessdot (b,1)$ with $b\in C$ are two left modular elements if and only if $b\in L$ is left modular and $b\in H(C)$. 
\end{lem}

\begin{lem}
\label{lemChainemaxIncomparable}
    All the maximal chains that do not intersect $C$ contain an element that is neither smaller than all the maximal elements of $C$ nor bigger than all its minimal elements. 
\end{lem}

\begin{lem}
\label{lemSpineCong}
The elements of the spine of an extremal congruence uniform lattice are left modular.
\end{lem}

Using \cref{thmCongruencenormal,lemSpineCong} we obtain

\begin{coro} 
\label{corocongruenceuniformExtremal}
Join-congruence uniform left modular lattices are join-extremal. 
For congruence uniform lattices, extremality and left modularity are equivalent.
\end{coro}

\begin{remark}
\cref{lemSpineCong} is a special case of combining results of \cite{thomas2005analoguedistributivityungradedlattices,Thomas_2019}, but is proved in our context by a simple induction.
The first statement of \cref{corocongruenceuniformExtremal} is a special case of Theorem $3.2$ of \cite{M_hle_2023}, as these lattices are JSD. Combining this with \cref{CorExtremalSD} gives, as was observed in \cite{M_hle_2023}, a generalization of the last statement of \cref{corocongruenceuniformExtremal} to all SD lattices.     
\end{remark}

\section{The $(P,\phi)$-Tamari lattices}
\label{sectionPphiTamari}

In this extended abstract we restrict ourselves to finite posets $P$ with a minimum $\hat{0}$ and chains $\phi: \phi(0)=\hat{0} < \phi(1) < \cdots < \phi(n-1)$ (an exception is made in \cref{subsectionChain}). In the sequel, let us fix such a poset $P$ and chain $\phi$ in $P$.

\subsection{Definition and first properties of $(P,\phi)$-Tamari}
\label{sectionGenerality}

For all $k<n$, denote $\mathcal{C}_k := I_P(\phi(k)) \times \{k\}$. The poset $\mathcal{C}_k$ is called the \defn{$k^{th}$ component} of the poset $\mathcal{C}_P^{\phi}:=\bigsqcup_{k<n} \mathcal{C}_k$.
We will later consider on $\mathcal{C}_P^{\phi}$ another partial order $\leq_{prod}$, called the product order, defined by $(x,i) \leq_{prod} (y,j)$ if and only if $x\leq y$ and $i\leq j$.
 We denote $a_{i,j}:=(\phi(j-i),j)$ for all $ i < j<n$. 
 In $\mathcal{C}_k$ these elements satisfy $a_{0,k}>a_{1,k}>\cdots > a_{k,k}$ and $a_{0,k}$ and $a_{k,k}$ are respectively the maximum and minimum elements of $\mathcal{C}_k$. 

\begin{defi}
\label{defConstruction}
Let $F$ be an order filter of $\mathcal{C}_P^{\phi}$. Then $F$ is \defn{torclosed} if for all $i$ and $j$ such that $i<j<n$, we have that $(x,i) \in F$ together with $(\phi(i+1),j) \in F$ implies that $(x,j)\in F$.

We call \defn{$(P,\phi)$-Tamari} the poset ordered by inclusion on the torclosed subsets of $\mathcal{C}_P^{\phi}$. Since the intersection of torclosed subsets is torclosed, it is a lattice with meet given by intersection, that we denote \defn{$\Tam{(P,\phi)}$}.
\end{defi}

\begin{remark}
See \cref{fig:firstExample} for an example and \cref{propTamari,propcompareTamari} for a justification of the name.
The name \emph{torclosed} for $F$ comes from the analogy with the torsion classes ; in special cases (see \cref{sectionAuslander}) $F$ being an order filter means closed by quotients, and the other condition means closed by extensions.
\end{remark}

\begin{figure}
\begin{subfigure}{.5\textwidth}
\centering
\begin{tikzpicture}
\begin{scope}[scale=0.9]
 \draw (0,0) -- (1,1) -- (0,2) -- (-1,1) -- (0,0);
\draw (0,0) node {$\bullet$};
\draw (-1,1) node {$\bullet$};
\draw (0,0) node [below ]{$\phi(0)$};
\draw (-1,1) node [left]{$\phi(1)$};
\draw (0,2) node {$\bullet$};
\draw (0,2) node [above ]{$\phi(2)$};   
\end{scope}

\begin{scope}[xshift= 1.8cm, scale= 1.2]
\node[draw] (7) at (0,0) {$7$};
\node[draw] (6) at (1,0) {$6$};
\node[draw] (5) at (1,1) {$5$};
\node[draw] (4) at (3,0) {$4$};
\node[draw] (3) at (2,1) {$3$};
\node (2) at (4,1) {$2$};
\node[draw] (1) at (3,2) {$1$};
\draw (6) -- (5);
\draw (4) -- (3) -- (1) -- (2) -- (4);
\draw[red] (7) --node[above] {$5$} (6);
\draw[red] (6) --node[above] {$1$} (4);
\draw[red] (5) -- node[above] {$1$} (3);
\draw[red] (7) to[bend right] node[below] {$3$} (4);
\end{scope}
\end{tikzpicture}
    \caption{On the left is a poset $P$ with the choice of a chain $\phi$. On the right are from left to right the components $\mathcal{C}_0$, $\mathcal{C}_1$ and $\mathcal{C}_2$ with the numbering from \cref{sectionCongruenceTopo} and with squares around the elements $a_{i,j}$.
    We added a horizontal red edge between $(x,i)$ and $(x,j)$ with label $(\phi(i+1),j)$.}
    \label{fig:posetfirstExample}
\end{subfigure}%
\begin{subfigure}{.5\textwidth}
\centering
    \begin{tikzpicture}
\begin{scope}[xshift=0.1cm,xscale=1.2]
\node[scale=0.7] (v17) at (0,7.5) {$\mathcal{C}_P^{\phi}$};
\node[scale=0.7] (v16) at (-2,4.4) {$5,6,7$};
\node[scale=0.7] (v15) at (2.5,5) {$1,2,3,4,7$};
\node[scale=0.7] (v14) at (2.5,3) {$1,2,7$};
\node[scale=0.7] (v13) at (2.5,2) {$1,7$};
\node[scale=0.7] (v12) at (2,1) {$7$};
\node[scale=0.7] (v11) at (0,6) {$1,2,3,4,5,6$};
\node[scale=0.7] (v10) at (-2,3) {$5,6$};
\node[scale=0.7] (v9) at (0.6,5) {$1,2,3,4,5$};
\node[scale=0.7] (v8) at (-0.3,4) {$1,2,3,5$};
\node[scale=0.7] (v7) at (-1,3) {$1,3,5$};
\node[scale=0.7] (v6) at (-2,1) {$5$};
\node[scale=0.7] (v5) at (1.3,4) {$1,2,3,4$};
\node[scale=0.7] (v4) at (0.6,3) {$1,2,3$};
\node[scale=0.7] (v3) at (-0.5,2) {$1,3$};
\node[scale=0.7] (v2) at (1.5,2.2) {$1,2$};
\node[scale=0.7] (v1) at (0,1) {$1$};
\node[scale=0.7] (v0) at (0,-0.2) {$\emptyset$};

\draw[->,>=latex] (v0)--(v1);
\draw[->,>=latex] (v0)--(v6);
\draw[->,>=latex] (v0)--(v12);
\draw[->,>=latex] (v1)--(v2);
\draw[->,>=latex] (v1)--(v3);
\draw[->,>=latex] (v1)--(v13);
\draw[->,>=latex] (v2)--(v4);
\draw[->,>=latex] (v2)--(v14);
\draw[->,>=latex] (v3)--(v4);
\draw[->,>=latex] (v3)--(v7);
\draw[->,>=latex] (v4)--(v5);
\draw[->,>=latex] (v4)--(v8);
\draw[->,>=latex] (v5)--(v9);
\draw[->,>=latex] (v5)--(v15);
\draw[->,>=latex] (v6)--(v7);
\draw[->,>=latex] (v6)--(v10);
\draw[->,>=latex] (v7)--(v8);
\draw[->,>=latex] (v8)--(v9);
\draw[->,>=latex] (v9)--(v11);
\draw[->,>=latex] (v10)--(v16);
\draw[->,>=latex] (v10)--(v11);
\draw[->,>=latex] (v11)--(v17);
\draw[->,>=latex] (v12)--(v16);
\draw[->,>=latex] (v12)--(v13);
\draw[->,>=latex] (v13)--(v14);
\draw[->,>=latex] (v14)--(v15);
\draw[->,>=latex] (v15)--(v17);
\draw[->,>=latex] (v16)--(v17);
\end{scope}
\end{tikzpicture}
\caption{$\Tam{(P,\phi)}$ from figure \ref{fig:posetfirstExample}.}
\label{fig:latticefirstExample}
\end{subfigure}
\caption{}
\label{fig:firstExample}
\end{figure}

\begin{prop}
     $\Tam{(P,\phi)}$ has a minimum element $\emptyset$ and a maximum element $\mathcal{C}_P^{\phi}$. Its atoms and coatoms are respectively $\{(\phi(k),k)\}$ and $\mathcal{C}_P^{\phi} \setminus \mathcal{C}_k$, for any $k<n$.
\end{prop}

\begin{lem}
\label{lemCovers}
    Let $b\lessdot c$ be a cover of $\Tam{(P,\phi)}$. Then there exists $k<n$ such that $c\setminus b \subseteq \mathcal{C}_k$ and it has a maximum element.
\end{lem}

\begin{prop}
\label{joinirrprop}
    The join-irreducibles of $\Tam{(P,\phi)}$ are the non-empty principal order filters of $\mathcal{C}_P^{\phi}$. Thus $|\,\mathrm{JIrr}(\Tam{(P,\phi)})\,| = |\mathcal{C}_P^{\phi}|$.
    The meet-irreducibles are in one to one correspondence with the pairs of elements $((x,k),a)$ where $(x,k)\in \mathcal{C}_P^{\phi}$ and $a=(\hat{0},k)$ or $a=a_{j,k}$ where $x$ is incomparable to $a_{j,k}$. For such a pair $((x,k),a)$, denote $l:=\mathrm{min}\{i\mid x\in I_P(\phi(i))\}$; then the meet-irreducible is
    $\displaystyle \Big(\bigcup_{i<l-1} \mathcal{C}_i \Big) \,\bigcup \,\Big( \bigcup_{l\leq i\leq k}\{(y,i) \in \mathcal{C}_P^{\phi} \mid y\not\leq x,\, (y,k)\not \leq a\} \Big) \,\bigcup \,\Big(\bigcup_{k<i<n} \mathcal{C}_i\Big)$.
\end{prop}

\begin{prop}
\label{propJSD}
$\Tam{(P,\phi)}$ is join-semidistributive. Moreover $\Tam{(P,\phi)}$ is semidistributive if and only if all the elements of $I_P(\phi(n-1))$ are comparable to all the elements of the chain $\phi$.
\end{prop}

\begin{proof}
    The first statement follows from \cref{lemCovers,lemEquiJSD}. The second statement from \cref{lemJSDjoinmeetSD,joinirrprop}.
\end{proof}

See \cref{fig:firstExample} for the smallest counter-example to semdistributivity.

\begin{prop}
\label{propspinechainsTam}
The spine and chains of longest length of $\Tam{(P,\phi)}$ correspond respectively to the order filters and linear extensions of the poset $(\mathcal{C}_P^{\phi},\leq_{prod})$.
Thus $\ell(\Tam{(P,\phi)})=|\mathcal{C}_P^{\phi}|$ and with \cref{joinirrprop} it is a join-extremal lattice. 
\end{prop}

\begin{prop}
\label{propcompareTamari}
    The induced subposet of $\Tam{(P,\phi)}$ on the torclosed subsets that are order filters generated by some of the $a_{i,j}$
    is a sublattice and is isomorphic to $\Tam_{n+1}$.
\end{prop}

\begin{ex}
    In \cref{fig:latticefirstExample}, forgetting the torclosed subsets $\{1,2\}$, $\{1,2,3\}$, $\{1,2,7\}$ and $\{1,2,3,5\}$ gives the sublattice $\Tam_4$.
\end{ex}

Denote $m_P^{\phi} := |\Tam{(P,\phi)}|$ and $I_i:=I_P(\phi(i))$ for all $i<n$.
\cref{propcompareTamari} gives $m_P^{\phi} \geq \frac{1}{n+2}\binom{2n+2}{n+1}$. Counting torclosed subsets with elements in only one component gives another lower bound $\sum_{i<n} |I_i| -(n-1)$. 
An obvious upper bound is $|I_0|\times |I_1| \times \cdots \times |I_{n-1}|$. 
If $\phi$ has one element then $m_P^{\phi}=2$ and if $\phi$ has $2$ elements then $m_P^{\phi}=2+|J(I_1)|$.

\begin{prop}
\label{propComptage3}
 If $\phi$ has $3$ elements then
$m_P^{\phi} = 1 + |J(I_1)| + \big|J\big(\,(\mathcal{C}_P^{\phi}\,,\,\leq_{prod})\,\big)\big| + |J(I_2\setminus I_1)|$.
\end{prop}

\begin{ex}
    For \cref{fig:latticefirstExample}, \cref{propComptage3} gives $m_P^{\phi} = 1+3+11+3=18$ elements. 
\end{ex}

\begin{question}
    Can we find a general formula for the number of elements of $\Tam{(P,\phi)}$ ? 
\end{question}

\subsection{Edge-labellings and congruences}
\label{sectionCongruenceTopo}

The choice of any linear extension $\mathcal{L}: x_1\succ x_2\succ \cdots \succ x_{|\mathcal{C}_P^{\phi}|}$ of $\big(\mathcal{C}_P^{\phi},\leq_{prod}\big)$ gives, by \cref{propspinechainsTam}, a longest chain 
$\psi_{\mathcal{L}}: \emptyset \lessdot \{x_1\} \lessdot \{x_1,x_2\}\lessdot \cdots \lessdot \{x_1,x_2,\dots,x_{|\mathcal{C}_P^{\phi}|}\} \lessdot \mathcal{C}_P^{\phi}$ of $\Tam(P,\phi)$. 
Moreover, $\mathcal{L}$ defines a numbering of the elements of $\mathcal{C}_P^{\phi}$ by assigning $i$ to the element $x_i$. Thus, by \cref{joinirrprop} we get a numbering of $\mathrm{JIrr}(\Tam(P,\phi))$; the join-irreducible generated by $x_i$ being denoted by $i$.
Using \cref{lemCovers}, we know (see \cref{background}) that we have an edge-labelling of $\Tam{(P,\phi)}$ that sends $b\lessdot c$ to the join-irreducible generated by $\mathrm{max}(c\setminus b)$. Thus we get an edge-labelling $\omega_{\mathcal{L}}$ defined by $\omega_{\mathcal{L}}~(b~\lessdot~c)~:=~i$ where $\mathrm{max}(c\setminus b)=x_i$. 
Then using the chain $\psi_{\mathcal{L}}$ in \cref{deflabelling}, we obtain

\begin{prop}
\label{propgamma}
We have $\gamma_1=\omega_{\mathcal{L}}$. We have $\gamma_2=\omega_{\mathcal{L}}$ if and only if $\mathcal{L}$ satisfies $(x,j)\succ (y,i)$, for all $j>i$, and if $(x,k)\not \leq a_{i,k}$ then $(x,k)\succ a_{i,k}$, for all $(x,k)\in \mathcal{C}_k$ and $i\leq k$. Such linear extensions exist.
\end{prop}

Fix a linear extension $\mathcal{L}$ as described in \cref{propgamma}.
See \cref{fig:posetfirstExample} for an example with the associated numbering of $\mathcal{C}_P^{\phi}$, and the edge-labelling $\omega_{\mathcal{L}}$ for \cref{fig:latticefirstExample} is $\mathrm{max}(c\setminus b)$ for covers $b\lessdot c$.
 From \cref{propgamma,thmlabelling,rmkLiu} we can obtain 

\begin{thm}
\label{thmleftmodular}
$\Tam(P,\phi)$ is left modular. The edge-labelling $\omega_{\mathcal{L}}$ is an $EL$-labelling such that any interval has at most one decreasing chain. Thus for all intervals $[x,y]$ of $\Tam(P,\phi)$, we have $\mu(x,y) \in \{-1,0,1\}$. 
\end{thm}

We now turn our attention to the congruences of $\Tam(P,\phi)$. The next result gives a characterization of the $D$ relation, where $(x,k)$ represents the join-irreducible generated by $(x,k)$. With \cref{propDayLW} it implies \cref{thmLW}.

\begin{prop}
\label{propDrelation}
We have $(x,i)\,D\,(y,i)$ if and only if there exists $k\leq i$ such that $y=\phi(k)$ and $y\not \leq x$. 
For all $i\neq j$, we have $(x,j)\,D\,(y,i)$ if and only if $i<j$, $y\leq x$, $ \phi(i+1) \not \leq x$ and there is no $y'\leq \phi(i)$ such that $y<y'\leq x$.
\end{prop}

\begin{ex}
In \cref{fig:posetfirstExample}, the $D$ relations between elements in the same component are $6D5, 4D3, 4D1, 3D1, 2D3,2D1$ and the others relations are $6D7, 4D7, 2D7, 4D6, 3D5,2D5$.
\end{ex}

\begin{thm}
\label{thmLW}
$\Tam(P,\phi)$ is join-congruence uniform.
\end{thm}

\begin{defi}
Let $K:=(K_0,K_1,\dots,K_{n-1})$ be a non-decreasing sequence of non-negative integers such that $K_i\leq i$ for all $i< n$. Denote $R_K :=\{(x,i)\in \mathcal{C}_P^{\phi}\mid x\geq \phi(K_i)\}$. On $\Tam(P,\phi)$ we define the \defn{Kupisch equivalence relation} $F\equiv_K F'$ if and only if $F\cap R_K = F'\cap R_K$.
\end{defi}

Using \cref{propDayLW,propDrelation} we obtain

\begin{prop}
\label{propKupisch}
A Kupisch equivalence relation is a lattice congruence of $\Tam(P,\phi)$.
\end{prop}

\section{Main examples}
\label{sectionMainExamples}

\subsection{The case of the chains}
\label{subsectionChain}

We assume that $P=C_n$ is a chain and $\phi$ is a chain of $C_n$. In this case, we also consider chains $\phi$ of $P$ such that $\phi(0)\neq 0$ and we can prove that we still get a lattice $\Tam(C_n,\phi)$.  
The components $\mathcal{C}_k$ for $k<n$ are themselves chains, and we identify a torclosed subset $F$ as the word on non-negative integers $u=u_1u_2\dots u_n$ where for all $i\in [n]$, $u_i:=|\mathcal{C}_{i-1} \cap F|$.  

\begin{lem}
\label{lemconditionchaintorclosed}
The word $u$ corresponds to a torclosed subset if and only if for all $i\in [n]$, $u_i\leq \phi(i-1)+1$, and for all $i\in [n]$ and $k\in [n-i]$, we have that $u_{i+k} > \phi(i+k) - \phi(i+1)$ together with $u_i\neq 0$ implies $u_{i+k} \geq u_i +\phi(i+k)-\phi(i)$.
\end{lem}

Let $p$ be a positive integer. For simplicity, we assume that $P=C_{np}$ and $\phi(i)=(i+1)\,p -1$ for all $i<n$.
For all $i\in [n]$, draw a segment from $(i-1,0)$ to $(i-1,u_i)$ in the Cartesian plane. Then \cref{lemconditionchaintorclosed} says that one can draw lines of slope $p$ passing through the $x$-axis and the top of each segment without crossing any segment.
For $p=1$ we recover a well-known description of the Tamari lattice due to Pallo.

\begin{prop} [\cite{10.1093/comjnl/29.2.171}]
\label{propTamari}
   $\Tam(C_n,C_n)=\Tam_{n+1}$.
\end{prop}

 Combe and Giraudo  \cite{Combe_2022} also obtained similar, but different, generalizations of the Tamari lattice called $\delta$-canyon lattices.
The tools that they have developed can be used in our model. We obtain

\begin{prop}
  $\Tam(C_n,\phi)$ is a congruence uniform left modular lattice.  
\end{prop}

\subsection{$d$-torsion classes of the higher Auslander algebras of type $\mathbb{A}$}
\label{sectionAuslander}

Let $d$ and $n$ be positive integers. Let $os_{n}^{d}$ be the set of non-decreasing $d$-sequences of elements of $\{0,1,\dots,n-1\}$. Let $\leq_{prod}$ be the product order on $os_{n}^{d}$.
The $d$-torsion classes of the higher Auslander algebras of type $\mathbb{A}$ are identified with the subsets $I  \subseteq os_n^{d+1}$ that satisfy conditions $(1)$ and $(2)$ of Theorem $5.13$ of \cite{AHJKPT23}.
Denote $L_n^d$ the lattice of these $d$-torsion classes ordered by inclusion. Using elementary techniques, we obtain

\begin{lem}
\label{newconditions}
Conditions $(1)$ and $(2)$ of Theorem $5.13$ of \cite{AHJKPT23} are equivalent to conditions $(1)$ and 
    \begin{itemize}
        \item[$(2') :$] $\forall j>i,\, \big(\,(x_1,x_2,\dots,x_d,i),(i+1,\dots,i+1,j)\,\big)\in I^2$ implies  $(x_1,x_2,\dots,x_d,j)\in I$.
    \end{itemize}
\end{lem}

\begin{thm}
\label{thmprincaus}
$L_n^d=\Tam\big((os_n^d,\leq_{prod}),\phi\big)$ with $\phi$ defined by $\phi(k)=(k,k,\dots,k)$ for all $k<n$.
\end{thm}

\begin{proof}
As they are both posets ordered by inclusion, we just have to prove that the $d$-torsion classes $I$ are the torclosed subsets.
Recall the definition of $\mathcal{C}_P^{\phi}$ in \cref{sectionGenerality}.
Condition~$(1)$ says that for any $i<n$, the subset $\{x\in I\,|\,x_{d+1}=i\}$ is an order filter of $\{x\in os_n^{d+1}\,|\,x_{d+1}=i\}$ for the product order (Remark $5.14$ of \cite{AHJKPT23}). Thus $I$ is an order filter of $\mathcal{C}_P^{\phi}$. Since $(\phi(i+1),j)=(i+1,\dots,i+1,j)$,
the result follows from \cref{newconditions}.   
\end{proof}

\begin{coro}
The lattice $L_n^d$ inherits all properties of $\Tam{(P,\phi)}$. In particular it is a join-semidistributive lattice, and is semidistributive if and only if $n\leq 2$ or $d=1$. The spine of $L_n^d$ is $J(os_n^{d+1})$. We have $\ell(L_n^d)=|\mathrm{JIrr}(L_n^d)|=\binom{n+d}{d+1}$ and
$$|\mathrm{MIrr}(L_n^d)|=\binom{n+d}{d+1}+\sum_{i=0}^{n-1} \,\,\sum_{j=0}^{i}\,\, \sum_{k=0}^{i-j-1}\,\, \sum_{l=1}^{j} \binom{d+i+l-j-k-2}{d-2} .$$
\end{coro}

\begin{prop}
\label{proptotallysym}
 We have an isomorphism of posets between $os_n^d$ and $os_{d+1}^{n-1}$, thus $|J(os_n^d)| = |J(os_{d+1}^{n-1})|$. We have that $|J(os_n^d)|$ is the number of totally symmetric $d$-dimensional partitions which fit in an $d$-dimensional box whose sides all have length $n$.  We deduce that $|J(os_3^d)|=2^{d+1}$ and $|J(os_4^d)|=a_{d+1}$ where 
 $a$ is sequence $A005157$ of $\mathrm{OEIS}$. 
\end{prop}

\begin{prop}
\label{proppropertiersLnd}
For all $d\geq 1$, $|L_3^d| = d+3+5\times 2^d$. Denote by $K_1:=[0\cdots022,13\cdots 3]$, $K_2:=[0\cdots 023,13\cdots 3]$ and $K_3:=\{x\in os_4^{d+1}\mid x\not\geq 1\cdots 13\}$ three subposets of $os_4^{d+1}$. Then
\begin{equation}
\label{formulaL4d}
|L_4^d|=8+d+3\times 2^{d+1} + (d+4)\,(a_d-1) + \sum_{i=2}^d a_i + 2\,\big(|J(K_1)|+|J(K_2)|\big) +|J(K_3)|    
\end{equation}
\end{prop}

\begin{remark}
With \cref{proptotallysym} and \eqref{formulaL4d}, we are able to compute $|L_4^7|=6543848$ and $|L_4^8|=130286256$. It completes column $n=4$ of Table $2$ of \cite{AHJKPT23}. 
\end{remark}

\begin{conjecture}
For $n$ fixed, $|L_n^d|=\mathcal{O}_{d\rightarrow \infty}\big(|J(os_n^{d+1})|\big)$. 
\end{conjecture}

\subsection{$d$-torsion classes of the higher Nakayama algebras of type $\mathbb{A}$}
\label{sectionNakayama}

We keep the notations from \cref{sectionAuslander}. A Kupisch series of type $\mathbb{A}_n$ is a sequence $\underline{l}=(l_0,l_1,\dots,l_{n-1})$ of positive integers satisfying $l_0=1$ and $\forall i\geq 1,\, 2\leq l_i\leq l_{i-1}+1$. We define 
$os_{\underline{l}}^{d+1} := \{y=(y_1,\dots,y_{d+1})\in os_n^{d+1} \mid y_1\geq y_{d+1}-l_{y_{d+1}} +1 \}  \subseteq os_n^{d+1}$.
We denote $L_{\underline{l}}^d$ the lattice ordered by inclusion of the $d$-torsion classes of the higher Nakayama algebra of type $\mathbb{A}$ associated to $\underline{l}$ (see Section $6.1$ of \cite{AHJKPT23} for the details). We can prove

\begin{lem}
\label{lemNaka}
 $L_{\underline{l}}^d$ is the lattice ordered by inclusion on the subsets $I\cap os_{\underline{l}}^{d+1} \subseteq os_{\underline{l}}^{d+1}$ for all torclosed subset $I\subseteq os_n^{d+1}$.
\end{lem}

\begin{thm}
    $L_{\underline{l}}^d$ is a lattice quotient of $L_n^d$. It inherits the properties of $\Tam{(P,\phi)}$ that are preserved by lattice quotient. In particular it is a join-semidistributive lattice.
\end{thm}

\begin{proof}
 Using \cref{lemNaka}, the elements of $L_{\underline{l}}^d$ identify with the equivalence classes of the relation $\equiv_N$ on $L_n^d$ defined by $I\equiv_N I'$ if and only if $I\cap os_{\underline{l}}^{d+1} = I'\cap os_{\underline{l}}^{d+1}$.
 Recall \cref{thmprincaus} that says that $L_n^d=\Tam\big((os_n^d,\leq_{prod}),\phi\big)$ with $\phi$ defined by $\phi(i)=(i,i,\dots,i)$ for all $i<n$.
 Then $\equiv_N$ is the Kupisch equivalence relation $\equiv_K$ on $L_n^d$ for $K=(K_0,K_1,\dots,K_{n-1})$ where $K_i:=\max(0,i-l_{i} +1)$ for all $i<n$. We conclude using \cref{propKupisch}.
\end{proof}

We conclude with an open question. 

\begin{question}
    Do the higher torsion classes of a $d$-cluster tilting subcategory always form a join-semidistributive lattice ?
\end{question}

%% if you use biblatex then this generates the bibliography
%% if you use some other method then remove this and do it your own way

\printbibliography

\end{document}